\documentclass[a4paper, 10pt, DIV10, headinclude=false, footinclude=false]{scrartcl}

\usepackage{amsthm, amsmath, amssymb, amsfonts}
\usepackage[utf8]{inputenc}
\usepackage[english]{babel}
\usepackage{url}
\usepackage{mathtools}

\usepackage{tikz}
\usetikzlibrary{patterns}

\numberwithin{figure}{section}
\numberwithin{table}{section}
\numberwithin{equation}{section}

\newtheorem{theorem}{Theorem}[section]
\newtheorem{lemma}[theorem]{Lemma}
\newtheorem{corollary}[theorem]{Corollary}
\newtheorem{proposition}[theorem]{Proposition}

\theoremstyle{definition}
\newtheorem{definition}[theorem]{Definition}
\newtheorem{remark}[theorem]{Remark}

\newenvironment{abstr}[1]{ \vspace{.05in}\footnotesize
       \parindent .2in
         {\upshape\bfseries #1. }\ignorespaces}{\par\vspace{.1in}}
\newenvironment{Abstract}{\begin{abstr}{Abstract}}{\end{abstr}}
\newenvironment{keywords}{\begin{abstr}{Key words}}{\end{abstr}}
\newenvironment{AMS}{\begin{abstr}{AMS subject classifications}}{\end{abstr}}

\DeclareMathOperator{\supp}{supp}
\DeclareMathOperator{\Int}{int}
\DeclareMathOperator{\diam}{diam}
\DeclareMathOperator{\HMM}{HMM}
\DeclareMathOperator{\LOD}{LOD}
\DeclareMathOperator{\ol}{ol}
\DeclareMathOperator{\loc}{loc}
\renewcommand{\Re}{\operatorname{Re}}

\newcommand{\halb}{\frac 12}
\newcommand{\nz}{\mathbb{N}}       
\newcommand{\rz}{\mathbb{R}}       
\newcommand{\pz}{\mathbb{P}}
\newcommand{\be}{\beta}

\newcommand{\de}{\delta}

\newcommand{\ep}{\varepsilon}

\newcommand{\om}{\omega}
\newcommand{\Om}{\Omega}

\newcommand{\si}{\sigma}

\newcommand\Ve{\mathbf{e}}

\newcommand\Vq{\mathbf{q}}

\newcommand\Vv{\mathbf{v}}
\newcommand\Vu{\mathbf{u}}
\newcommand\Vw{\mathbf{w}}

\newcommand\Vz{\mathbf{z}}

\newcommand\VI{\mathbf{I}}

\newcommand\VQ{\mathbf{Q}}

\newcommand\VV{\mathbf{V}}
\newcommand\VW{\mathbf{W}}
\newcommand\Veta{\boldsymbol{\eta}}

\newcommand\Vpsi{\boldsymbol{\psi}}
\newcommand\Vphi{\boldsymbol{\phi}}

\newcommand\CB{\mathcal{B}}

\newcommand\CH{\mathcal{H}}
\newcommand\CI{\mathcal{I}}

\newcommand\CK{\mathcal{K}}

\newcommand\CR{\mathcal{R}}
\newcommand\CS{\mathcal{S}}
\newcommand\CT{\mathcal{T}}

\newcommand\UN{\textup{N}}

\newcommand\VCI{\boldsymbol{\mathcal{I}}}

\allowdisplaybreaks[4]

\begin{document}

\title{Localized Orthogonal Decomposition for two-scale Helmholtz-type problems%
\thanks{This work was supported by the Deutsche Forschungsgemeinschaft (DFG) in the project ``OH 98/6-1: Wellenausbreitung in periodischen Strukturen und Mechanismen negativer Brechung''} 
}

\author{Mario Ohlberger%
\thanks{Angewandte Mathematik: Institut f\"ur Analysis und Numerik, Westf\"alische Wilhelms-Uni\-ver\-si\-t\"at M\"unster, D-48149 M\"unster
}
\and Barbara Verf\"urth\footnotemark[2]}

\date{}

\maketitle

\begin{Abstract}
In this paper, we present a Localized Orthogonal Decomposition (LOD) in Petrov-Galerkin formulation for a two-scale Helmholtz-type problem. 
The two-scale problem is, for instance, motivated from the homogenization of the Helmholtz equation with high contrast, studied together with a corresponding multiscale method in (Ohlberger, Verf\"urth. A new Heterogeneous Multiscale Method for the Helmholtz equation with high contrast, arXiv:1605.03400, 2016).
There, an unavoidable resolution condition on the mesh sizes in terms of the wave number has been observed, which is known as ``pollution effect'' in the finite element literature.
Following ideas of (Gallistl, Peterseim. \textit{Comput.\ Methods Appl.\ Mech.\ Engrg.} 295:1-17, 2015),  we use standard finite element functions for the trial space, whereas the test functions are enriched by solutions of subscsale problems (solved on a finer grid) on local patches. Provided that the oversampling parameter $m$, which indicates the size of the patches, is coupled logarithmically to the wave number, we obtain a quasi-optimal method under a reasonable resolution of a few degrees of freedom per wave length, thus overcoming the pollution effect.
In the two-scale setting, the main challenges for the LOD lie in the coupling of the function spaces and in the periodic boundary conditions.
\end{Abstract}



\section{Introduction}
The numerical solution of high frequency Helmholtz problems with standard finite element methods is still a very challenging task due to the oscillatory nature of the solutions.
This manifests itself in the so-called pollution effect \cite{BS00pollutionhelmholtz}: A much smaller mesh size than needed for a meaningful approximation of the solution are required for the stability and convergence of the numerical scheme. Typically, this leads to a resolution condition $k^\alpha H=O(1)$ with $\alpha>1$ instead of $kH=O(1)$,  i.e. few degrees of freedom per wave length. 
The challenge becomes even greater when additionally studying the Helmholtz problem in a  heterogeneous medium, such as locally periodic structures.

(Locally) periodic media, such as photonic crystals, can exhibit astonishing properties such as band gaps, artificial magnetism, or negative refraction \cite{EP04negphC,CJJP02negrefraction,OBP02magneticactivity,PE03lefthanded}. 
One popular and successful modelling setup are scatterers made up of two materials with high contrast in the permittivity \cite{BBF09hom3d,BF04homhelmholtz,BS10splitring,BS13plasmonwaves,CC15hommaxwell,BF97homfibres,LS15negindex}.
As the materials' fine-scale structures are much smaller than the wavelength, homogenization and corresponding numerical multiscale methods, such as the Heterogeneous Multiscale Method (HMM), are efficient tools to reduce the problem's complexity.
Homogenization theory gives an effective Helmholtz problem, which describes the macroscopic behavior of the original problem, but involves the solution of some additional cell problems to determine the effective material parameters.
This procedure can also be coupled in the so-called \emph{two-scale equation}, which gives the macroscopic and the cell problems in one variational problem.
The Heterogeneous Multiscale Method can be seen as a Galerkin discetization of this two-scale equation and it gives good approximations even of the solution in the heterogeneous medium, see  \cite{OV16hmmlod1}.
However, as every standard Galerkin method, the HMM for Helmholtz-type problems also suffers from the pollution effect described in the beginning \cite{EM12helmholtz,MS11helmholtz,Sau06convanahelmholtz}.
There are several attempts to reduce this pollution effect, e.g.\ high-order finite element methods \cite{EM12helmholtz,MS11helmholtz}, (hybridizable) discontinuous Galerkin methods \cite{CLX13HDGhelmholtz,GM11HDGhelmholtz}, or (plane wave) Trefftz methods \cite{HMP16surveyTrefftz,HMP15pwdg,PPR15pwvem}.

Recently, the works  \cite{BGP15hethelmholtzLOD,GP15scatteringPG,P15LODhelmholtz} suggested a multiscale Petrov-Galerkin method for the Helmholtz equation to reduce the pollution effect.
The method is based on a so called Localized Orthogonal Decomposition (LOD), as first introduced in \cite{MP14LODelliptic}.
The LOD builds on ideas from the Variational Multiscale Method (VMM)  \cite{HFMQ98VMM,HS07VMM} by splitting the solution space into a coarse and a fine part. 
The coarse standard finite element functions are modified by adding a correction from the fine space, which is constructed as the kernel of an interpolation operator.
The corrections are problem dependent and computed by solving PDEs on a finer grid. 
In most cases the corrections show exponential decay, which justifies to truncate them to patches of coarse elements.
Since its development, the LOD has been successfully applied to elliptic boundary problems \cite{HM14LODbdry}, eigenvalue problems \cite{MP15LODeigenvalues}, mixed problems \cite{HHM15LODmixed}, parabolic problems \cite{MP15LODeigenvalues}, the wave equation \cite{AH14LODwave} or elasticity \cite{HP16LODeleasticity}.
A review is given in \cite{Pet15LODreview}.
As already discussed in \cite{EGH15LODpetrovgalerkin} and analyzed in more detail in \cite{Pet15LODreview}, Petrov-Galerkin formulations show the same stability and convergence behavior as the symmetric Galerkin methods while being less expensive with respect to communication. 
An extensive discussion on implementation aspects of the LOD, such as how to exploit a priori known structures to reduce the number of local subscale problems, is given in \cite{EHMP16LODimplementation} and some remarks for acoustic scattering in \cite{GP15scatteringPG}.

In this article, we present how the LOD can be applied to two-scale Helmholtz-type problems.
Following the ideas in \cite{GP15scatteringPG}, the test functions are modified by local subscale correction, which are solved on a fine grid fulfilling the resolution condition for standard Galerkin finite element methods.
The resulting test functions have support over patches of size $mH_c$ and $mh_c$; $H_c$ being the mesh size of the grid for the macroscopic domain $G$, $h_c$ being the mesh size of the grid for the unit square $Y$ and $m$ being the adjustable oversampling parameter.
Under the condition that $m\approx \log(k)$ and $k(H_c+h_c)\lesssim1$, the (two-scale) LOD is stable and quasi-optimal, i.e.\ the error between the LOD-approximation and the analytical solution of the two-scale equation is of the order of the best-approximation error.
These are the results expected from the one-scale setting.
The main contribution is the rigorous (theoretical) analysis for this LOD in the two-scale setting.
The novelty here is that first, we have to deal with coupled functional spaces and sesquilinear forms.
This coupling has to be taken care of in all new definitions of (again coupled) spaces and also in the estimates, where we have to jump back and forth between the two-scale norms and the properties of each individual space and interpolation operator.
Second, our coupled two-scale functional spaces involve periodic boundary conditions on the unit cube, which also have to be paid attention at when defining the interpolation operators and the oversampling patches at the boundary.

Finally, let us summarize the connection of the LOD in two-scale setting to original Helmholtz problem with high contrast, but also emphasize the general nature of our (theoretical) findings.
Combining the HMM of \cite{OV16hmmlod1} and the LOD of the present paper, we have  solved the Helmholtz problem with high contrast by a double approximation procedure: First removing the challenges from the fine-scale structures related to the periodicity $\de$ and then reducing the pollution effect related to the wave number $k$, where we recall the three-scale nature $\delta\ll k^{-1}<1$.
Our numerical experiments in \cite{OV16hmmlod1} have not experienced great restriction from the resolution condition, but more sophisticated examples or even higher wave numbers may make the application of the LOD favorable, which hereby has the necessary theoretical footing.
Concerning possible generalizations of our work, let us note that we concentrate on the two-scale setting described in \cite{OV16hmmlod1}, but the analysis can also be adapted to the case without high contrast\cite{CS14hmmhelmholtz} and might be useful for future numerical studies of the three-dimensional cases\cite{BBF09hom3d,BS10splitring,BS13plasmonwaves,CC15hommaxwell,LS15negindex}.
Moreover, the definition of the patches and interpolation operators in the periodic case and also the coupling and decoupling of various function spaces may be of general interest.

The paper is structured as follows: Section \ref{sec:setting} introduces the notation and problem setting, see also \cite{OV16hmmlod1}. In Section \ref{sec:LOD}, the LOD in Petrov-Galerkin formulation is defined in the two-scale setting, where we also give some general comments on implementation aspects. Stability and convergence of the error are discussed in Section \ref{sec:error}.  Section \ref{sec:decay} is devoted to the detailed proof of the decay of the correctors.

\section{The two-scale Helmholtz problem}
\label{sec:setting}
In this section, we introduce what we call the two-scale Helmholtz problem and the necessary notation.
For further details on the derivation of this two-scale model and its practical relevance we refer to \cite{OV16hmmlod1}.

For the remainder of this article, let $\Om\subset\subset G\subset \rz^2$ be two domains with (polygonal) Lipschitz boundary at least. 
Throughout this paper, we use standard notation for Lebesgue and Sobolev spaces: For a domain $\om$, $p\in [1, \infty)$, and $s\in \rz_{\geq 0}$, $L^p(\om)$ denotes the complex Lebesgue space and $H^s(\om)$ denotes the complex (fractional) Sobolev space. 
The dot denotes a normal (real) scalar product, for a complex scalar product we explicitly conjugate the second component by using $v^*$ as the conjugate complex of $v$. 
The complex $L^2$ scalar product on a domain $\om$ is abbreviated by $(\cdot, \cdot)_\om$ and the domain is omitted if no confusion can arise.
For $v\in H^1(\om)$, we frequently use the $k$-dependent norm
\[\|v\|_{1,k,\om}:=(\|\nabla v\|^2_\om+k^2\|v\|^2_\om)^{1/2},\]
which is obviously equivalent to the $H^1$ norm.
We write $Y:=[-\halb, \halb)^2$ to denote the two-dimensional unit square. 
We indicate $Y$-periodic functions\cite{All92twosc} with the subscript $\sharp$. For example, $H^1_{\sharp, 0}(Y)$ is the space of $Y$-periodic functions from $H^1_{\loc}(\rz^2)$ with zero average over $Y$.
For $Y^*\subset Y$, we denote by $H^1_{\sharp, 0}(Y^*)$ the restriction of functions in $H^1_{\sharp, 0}(Y)$ to $Y^*$. 
For $D\subset \subset Y$, $H^1_0(D)$ can be interpreted as subspace of $H^1_\sharp(Y)$ and we write $H^1_0(D)_\sharp$ to emphasize this periodic extension. By $L^p(\Om; X)$ we denote Bochner-Lebesgue spaces over the Banach space $X$.

The two-scale Helmholtz problem is now formulated as follows: 
We seek $\Vu:=(u, u_1, u_2)\in \CH$ such that
\begin{equation}
\label{eq:twoscaleeq}
\CB(\Vu, \Vpsi)=\int_{\partial G}g\psi^*\, d\si\qquad\forall \Vpsi:=(\psi, \psi_1, \psi_2)\in \CH
\end{equation}
with the two-scale function space $\CH:=H^1(G)\times L^2(\Om; H^1_{\sharp,0}(Y^*))\times L^2(\Om; H^1_0(D)_\sharp)$ and the two-scale sesquilinear form $\CB$ defined by
\begin{equation*}
\begin{split}
&\!\!\!\!\CB(\Vv, \Vpsi)\\
&:=\int_\Om \int_{Y^*}\!\ep_e^{-1}(\nabla v+\nabla_y v_1)\cdot (\nabla\psi^*+\nabla_y \psi_1^*)\, dy dx
+\int_G\int_D \ep_i^{-1}\nabla_y v_2\cdot \nabla_y \psi_2^*\, dydx\\
&\quad-k^2\int_G\int_Y (v+\chi_D v_2)(\psi^*+\chi_D \psi_2^*)\, dydx+\int_{G\setminus \overline{\Om}}\nabla v\cdot \nabla \psi^*\, dx-ik\int_{\partial G} v\psi^*\, d\si.
\end{split}
\end{equation*}
Beside the natural norm on $\CH$ it is convenient for error estimation to define the following two-scale energy norm
\begin{align}
\label{eq:errornorm}
\|(v, v_1, v_2)\|^2_{e, \om\times R}:=\|\nabla v +\nabla_y v_1\|^2_{\om\times R_1}+\|\nabla_y v_2\|^2_{\om\times R_2}+k^2\|v+\chi_D v_2\|^2_{\om\times R}
\end{align}
for a subdomain $\om\times R\subset G\times Y$ with $R_1:=R\cap Y^*$ and $R_2:=R\cap D$.
Furthermore, we introduce a version of the $H^1$ semi-norm on $\CH$ via
\begin{equation}
\label{eq:errorh1semi}
\|(v, v_1, v_2)\|^2_{1,e, \om \times R}:=\|\nabla v+\nabla_y v_1\|^2_{\om\times R_1}+\|\nabla_y  v_2\|^2_{\om \times R_2},
\end{equation}
which is induced by the sesquilinear form
\[(\Vv, \Vpsi)_{1,e, \om \times R}:=(\nabla v+\nabla_y v_1, \nabla \psi_1+\nabla_y \psi_2)_{\om\times R_1}+(\nabla_y  v_2, \nabla_y \psi_2)_{\om \times R_2}\qquad \forall \Vv, \Vpsi\in \CH.\]
We omit the subscript for the subdomain if $\om\times R=G\times Y$.
In \cite{OV16hmmlod1}, we proved the following Lemma.

\begin{lemma}
\label{lem:propertiesCB}
There exist constants $C_B>0$ and $C_{\min}:=\min\{1, \ep_e^{-1}, \Re(\ep_i^{-1})\}>0$ depending only on the parameters and the geometry, such that $\CB$ is continuous with constant $C_B$ and fulfills a G{\aa}rding inequality with constant $C_{\min}$, i.e.\
\begin{align*}
|\CB(\Vv, \Vpsi)|\leq C_B \|\Vv\|_e\|\Vpsi\|_e\quad \text{and}\quad \Re \CB(\Vv, \Vv)+2k^2\|v+\chi_D v_2\|_{G\times Y}^2\geq C_{\min}\|\Vv\|_e^2.
\end{align*}
\end{lemma}
Note that $C_{\min}$ can also be used to estimate the gradient terms in $\CB$ by $C_{\min}\|\cdot\|_{1,e}$ from below.
Furthermore, the unique solution $\Vu\in \CH$ to \eqref{eq:twoscaleeq} (with additional volume term $f\in L^2(G)$ on the right-hand side) fulfills the following stability estimate
\begin{equation}
\label{eq:stabiltwosc}
\|\Vu\|_e\leq C k^q(\|f\|_{L^2(G)}+\|g\|_{H^{1/2}(\partial G)}) \qquad \text{for some}\quad q\in \nz_0,
\end{equation}
see \cite{OV16hmmlod1} for a proof with $q=3$.

For the error analysis, we will compare the solution of the Localized Orthogonal Decomposition to a discrete reference solution, which is only needed for the theory and will never be computed in practical implementations.
We introduce conforming and shape regular triangulations $\CT_H$ and $\CT_h$ of $G$ and $Y$, respectively. Additionally, we assume that $\CT_H$ resolves the partition into $\Om$ and $G\setminus \overline{\Om}$ and that $\CT_h$ resolves the partition into $D$ and $Y^*$ and is periodic in the sense that it can be wrapped to a regular triangulation of the torus (without hanging nodes). 
We use the conforming subspace $\VV_{H,h}\subset \CH$ made up of linear Lagrange elements via $\VV_{H,h}:=V_H^1\times L^2(\Om; \widetilde{V}_h^1(Y^*))\times L^2(\Om; V_h^1(D))$, where the $L^2$ spaces with respect to $x$ can be additionally approximated by piece-wise polynomials.
The discrete reference solution $\Vu_{H,h}=(u_H, u_{h,1}, u_{h,2})\in\VV_{H, h}$ is the solution to
\begin{equation}
\label{eq:discrtwoscaleeq}
\CB(\Vu_{H,h},\Vpsi_{H,h})=\int_{\partial G}g\psi_H^*\, d\si\qquad \forall \Vpsi_{H,h}:=(\psi_H, \psi_{h,1}, \psi_{h,2})\in \VV_{H,h}.
\end{equation}
We assume that this direct discretization is stable in the following sense:
The (fine) mesh sizes $H$ and $h$ are small enough (in dependence on the wave number $k$) that there is a constant $C_{\HMM}$ such that
\begin{equation}
\label{eq:stabHMM}
(C_{\HMM}\,k^{q+1})^{-1}\leq \inf_{\Vv_{H,h}\in \VV_{H,h}}\sup_{\Vpsi_{H,h}\in \VV_{H,h}}\frac{\Re \CB(\Vv_{H,h}, \Vpsi_{H,h})}{\|\Vv_{H,h}\|_e\|\Vpsi_{h,h}\|_e}.
\end{equation}
In \cite{OV16hmmlod1}, we discussed that this direct discretization can be re-cast into the traditional formulation of a Heterogeneous Multiscale Method and we proved that the discrete inf-sup-condition \eqref{eq:stabHMM} holds under the classical resolution condition $k^{q+2}(H+h)=O(1)$.

\begin{remark}
We demonstrate the LOD at the specific example at the two-scale Helmholtz problem obtained in \cite{OV16hmmlod1}.
However, the theory can easily be extended to more general two-scale Helmholtz problems, which fulfill the following assumptions
\begin{itemize}
\item the variational problem \eqref{eq:twoscaleeq} involves a continuous sesquilinear form with G{\aa}rding inequality, i.e.\ an analogue of Lemma \ref{lem:propertiesCB};
\item the analytical solution fulfills a stability estimate \eqref{eq:stabiltwosc};
\item the (direct) Galerkin discretization \eqref{eq:discrtwoscaleeq} is stable \eqref{eq:stabHMM}.
\end{itemize}
\end{remark}

\section{The Localized Orthogonal Decomposition}
\label{sec:LOD}
In this section, we introduce the notation on meshes, finite element spaces, and (quasi)-inter\-pola\-tion operators and define the  Localized Orthogonal decomposition in Petrov-Galerkin formulation for the two-scale setting. We close with some remarks regarding an implementation of the two-scale LOD.

\subsection{Meshes and finite element spaces}
Let the (fine) meshes $\CT_H$ of $G$ and $\CT_h$ of $Y$ be given as in the previous section, we assume that $H$ and $h$ are small enough that \eqref{eq:stabHMM} is fulfilled.
We consider a second, coarse discretization scale $H_c>H$ and $h_c>h$: Let $\CT_{H_c}$ and $\CT_{h_c}$ denote corresponding conforming, quasi-uniform, and shape regular triangulations of $G$ and $Y$, respectively. 
As for the fine grids, we additionally assume that $\CT_{h_c}$ is periodic and that $\CT_{H_c}$ and $\CT_{h_c}$ resolve the partition of $G$ into $\Om$ and its complement and of $Y$ into $D$ and $Y^*$, respectively. We denote by $\CT_{h_c}(Y^*)$ and $\CT_{h_c}(D)$ the parts of $\CT_{h_c}$ triangulating $Y^*$ and $D$, respectively.
The global mesh sizes are defined as $H_c:=\max\{ \diam(T)|T\in \CT_{H_c}\}$ and $h_c:=\max\{\diam(S)|S\in \CT_{h_c}\}$.
For the sake of simplicity we assume that $\CT_H$ and $\CT_h$ are derived from $\CT_{H_c}$ and $\CT_{h_c}$, respectively, by some regular, possibly non-uniform, mesh refinement including at least one global refinement.
We consider simplicial partitions, but the theory of this paper carries over to quadrilateral  partitions \cite{GP15scatteringPG} and even meshless methods would be possible \cite{HMP14partunity}.

\begin{figure}
\label{fig:patch}
\begin{center}
\begin{tikzpicture}
\path[fill=lightgray] (3,0)--(4,1)--(3.5,1.5)--(4,2)--(2,4)--(0,2)--(0.5, 1.5)--(0,1)--(1,0)--cycle;
\path[fill=gray] (2,0)--(3,1)--(2.5, 1.5)--(3,2)--(2,3)--(1,2)--(1,1)--cycle;
\path[fill=black] (1.5, 1.5)--(2,2)--(2,1)--cycle;
\path[pattern=horizontal lines] (2, 4)--(1,4)--(1.5, 3.5)--cycle;
\path[pattern=horizontal lines] (2, 4)--(3,4)--(2.5, 3.5)--cycle;
\draw(0,0) rectangle (4,4);
\draw(1,0)--(1,4);
\draw(2,0)--(2,4);
\draw(3,0)--(3,4);
\draw(0,1)--(4,1);
\draw(0,2)--(4,2);
\draw(0,3)--(4,3);
\draw (0,0)--(4,4);
\draw (0,1)--(3,4);
\draw(0,2)--(2,4);
\draw(0,3)--(1,4);
\draw(1,0)--(4,3);
\draw(2,0)--(4,2);
\draw(3,0)--(4,1);
\draw(0,4)--(4,0);
\draw(0,3)--(3,0);
\draw(0,2)--(2,0);
\draw(0,1)--(1,0);
\draw(1,4)--(4,1);
\draw(2,4)--(4,2);
\draw(3,4)--(4,3);
\end{tikzpicture}
\end{center}
\caption{Triangle $T$ (in black) and its first and second order patches (additional elements for $\UN(T)$ in dark gray and additional elements for $\UN^2(T)$ in light gray). Striped triangles belong to $\UN^2(T)$ in the case of periodic boundary conditions.}
\end{figure}
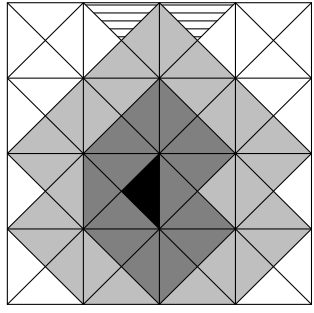

Given any  subdomain $\om\subset \overline{G}$ define its neighborhood via
\[\UN(\om):=\Int(\cup\{T\in \CT_{H_c}|T\cap\overline{\om}\neq \emptyset\})\]
and for any $m\geq 2$ the patches
\[\UN^1(\om):=\UN(\om)\qquad \text{and}\qquad\UN^m(\om):=\UN(\UN^{m-1}(\om)),\]
see Figure \ref{fig:patch} for an example.
The shape regularity implies that there is a uniform bound $C_{\ol, m, G}$ on the number of elements in the $m$-th order patch
\[\max_{T\in \CT_{H_c}}\sharp\{K\in \CT_{H_c}|K\subset\overline{\UN^m(T)}\}\leq C_{\ol, m, G}\]
and the quasi-uniformity implies that $C_{\ol, m, G}$ depends polynomially on  $m$. We abbreviate $C_{\ol, G}:=C_{\ol, 1, G}$. 
The patches can also be defined in a similar way for a subdomain $R\subset\overline{Y}$.
Here, we split $R=R_1\cup R_2$ with $R_1=R\cap D$ and $R_2=R\cap Y^*$, where $R_1$ or $R_2$ may be empty, and we write in short $\UN^m(R):=\UN^m(R_1)\cup \UN^m(R_2)$.
$\UN^m(R_1)$ is defined in the same way as before, in particular, it ends at the boundary $\partial D$.
For the patch $\UN^m(R_2)$ we interpret $\overline{Y^*}$ as part of the torus. This implies that $\UN^m(R_2)$ ends at the inner boundary $\partial D$, but is continued periodically over the outer boundary $\partial Y$.
This means that also the striped triangles in Figure \ref{fig:patch} belong to the second patch for the periodic setting.
We denote the overlap constants by $C_{\ol, m, Y}$ and $C_{\ol, Y}$. 
By slight abuse of notation, we write $\UN^m(\om\times R):=\UN^m(\om)\times \UN^m(R)$ for a subdomain $\om\times R\subset\overline{G}\times \overline{Y}$.

We denote the conforming finite element triple space consisting of lowest order Lagrange elements with respect to the meshes $\CT_{H_c}$ and $\CT_{h_c}$ by $\VV_{H_c, h_c}$ as in the previous section.
Again, we have $\VV_{H_c, h_c}:=V_{H_c}^1\times L^2(\Om; \widetilde{V}_{h_c}^1(Y^*))\times L^2(\Om; V_{h_c}^1(D))$ and we moreover note that $\VV_{H_c, h_c}\subset \VV_{H,h}\subset \CH$.

\subsection{Quasi-interpolation}
A key tool in the definition and the analysis is a bounded linear surjective  (quasi)-interpolation operator $\VI_{H_c, h_c}:\VV_{H,h}\to \VV_{H_c, h_c}$ that acts as a stable quasi-local projection in the following sense: 
It is a projection, i.e.\ $\VI_{H_c, h_c}\circ\VI_{H_c, h_c}=\VI_{H_c, h_c}$, and it is constructed as $\VI_{H_c, h_c}:=(I_{H_c}, I_{h_c}^{Y^*}, I_{h_c}^D)$, where each (quasi)-interpolation operator fulfills the following.
There exist  constants $C_{I_{H_c}}$, $C_{I_{h_c}^{Y^*}}$, and $C_{I_{h_c}^D}$ such that for all $\Vv_{H, h}:=(v_H, v_{h,1}, v_{h,2})\in \VV_{H,h}$ and for all $T\in \CT_{H_c}$, $S_1\in \CT_{h_c}(Y^*)$ and $S_2\in \CT_{h_c}(D)$
\begin{equation}
\label{eq:intpol}
\begin{split}
H_c^{-1}\|v_H-I_{H_c}(v_H)\|_T+\|\nabla I_{H_c}(v_H)\|_T&\leq C_{I_{H_c}}\|\nabla v_H\|_{\UN(T)},\\
h_c^{-1}\|v_{h,1}-I_{h_c}^{Y^*}(v_{h,1})\|_{T\times S_1}+\|\nabla_y  I_{h_c}^{Y^*}(v_{h,1})\|_{T\times S_1}&\leq C_{I_{H_c}^{Y^*}}\|\nabla_y v_{h,1}\|_{T\times \UN(S_1)},\\
h_c^{-1}\|v_{h,2}-I_{h_c}^D(v_{h,2})\|_{T\times S_2}+\|\nabla_y I_{h_c}^D(v_{h,2})\|_{T\times S_2}&\leq C_{I_{H_c}^D}\|\nabla_y v_{h,2}\|_{T\times \UN(S_2)}.
\end{split}
\end{equation}
We abbreviate $C_{\VI}:=\max\{C_{I_{H_c}}, C_{I_{h_c}^{Y^*}}, C_{I_{H_c}^D}\}$.
Under the mesh condition that $k(H_c+h_c)\lesssim 1$, this implies stability in the two-scale energy norm
\begin{equation}
\label{eq:intpolenergystable}
\|\VI_{H_c, h_c}\Vv_{H,h}\|_e\leq C_{\VI, e}\|\Vv_{H,h}\|_e\qquad \forall \Vv_{H,h}\in \VV_{H,h}.
\end{equation}

The quasi-interpolation operator $\VI_{H_c, h_c}$ is not unique: A different choice might lead to a different Localized Orthogonal Decomposition and this  can even affect the practical performance of the method\cite{P15LODhelmholtz}.
One popular choice is the concatenation of the $l^2$ projection onto piece-wise polynomials and the Oswald interpolation operator.
Other choices are discussed in \cite{EHMP16LODimplementation,Pet15LODreview}.
For the operators $I_{h_c}^{Y^*}$ and $I_{h_c}^D$ not that they only act with respect to the second variable $y$.
For $I_{h_c}^{Y^*}$, one can preserve periodicity as follows: The averaging process of the Oswald interpolation operator has to be continued over the periodic boundary (as for the patches before).

\subsection{Definition of the LOD}
The method approximates the discrete two-scale solution $\Vu_{H,h}:=(u_H, u_{h,1}, u_{h,2})$ to \eqref{eq:discrtwoscaleeq} for given (fine) mesh sizes $H$, $h$. It is determined by the choice of the coarse mesh sizes $H_c$ and $h_c$ and the oversampling parameter $m$ explained in the following.
We assign to any $(T, S_1, S_2)\in \CT_{H_c}\times \CT_{h_c}(Y^*)\times \CT_{h_c}(D)$ its $m$-th order patch $G_T\times Y^*_S\times D_S:=\UN^m(T)\times \UN^m(S_1)\times \UN^m(S_2)$ and define for any $\Vv_{H,h}=(v_H, v_{h,1}, v_{h,2}), \Vpsi_{H,h}=(\psi_H, \psi_{h,1}, \psi_{h,2})\in \VV_{H,h}$ the localized sesquilinear form
\begin{align*}
&\!\!\!\!\CB_{G_T\times Y_S}(\Vv_{H,h}, \Vpsi_{H,h})\\*
&:=(\ep_e^{-1}(\nabla v_H+\nabla_y v_{h,1}), \nabla \psi_H+\nabla_y \psi_{h,1})_{(G_T\cap \Om)\times Y^*_S}+(\ep_i^{-1}\nabla_y v_{h,2}, \nabla_y \psi_{h,2})_{G_T\times D_S}\\*
&\qquad+(\nabla v_H, \nabla \psi_H)_{G_T\cap(G\setminus \overline{\Om})}
-k^2(v_H+\chi_Dv_{h,2}, \psi_H,+\chi_D\psi_{h,2})_{G_T\times Y_S}\\*
&\qquad-ik(v_H, \psi_H)_{\partial G_T\cap\partial G}
\end{align*}
with $Y_S:=D_S\cup Y^*_S$. For $m=0$ (i.e.\ $\UN^m(T)=T$), we write $\CB_{T\times S}$ with $S=S_1\cup S_2$.
Note that the oversampling parameter does not have to be the same for $G$, $Y^*$, and $D$. We could as well introduce patches $\UN^{m_0}(T)\times \UN^{m_1}(S_1)\times \UN^{m_2}(S_2)$, but we choose $m_0=m_1=m_2=:m$ for simplicity of presentation and to improve readability.

We define the (truncated) finite element functions on the fine-scale meshes as
\begin{align*}
V_H(G_T)&:=\{v_H\in V_H^1|v_H=0 \text{ outside }G_T\},\\
L^2(\Om_T, \widetilde{V}_h^1(Y^*_S))&:=\{v_{h, 1}\in L^2(\Om; \widetilde{V}_h^1(Y^*))|v_{h,1}=0\text{ outside }(G_T\cap \Om)\times (Y^*)_S\}
\end{align*}
and $L^2(\Om_T; V_h^1(D_S))$ in a similar way.
Define the null space
\begin{align*}
&\!\!\!\!\VW_{H,h}(G_T\times Y_S)\\
&:=\{\Vw_{H,h}\in V_H(G_T)\times L^2(\Omega_T; \tilde{V}_h(Y^*_S))\times L^2(\Omega_T; V_h(D_S))|\VI_{H_c, h_c}(\Vw_{H,h})=0\}
\end{align*}
and note that $\VW_{H,h}(G_T\times Y_S):=W_H(G_T)\times L^2(\Om; W_h(Y^*_S))\times L^2(\Om; W_h(D_S))$, where $W_H$ and $W_h$ are defined as the kernels of the corresponding (single) interpolation operators $I_{H_C}$ and $I_{h_c}^{Y^*}$ and $I_{h_c}^D$, respectively.
For given $\Vv_{H_c, h_c}\in \VV_{H_c, h_c}$ we define the localized correction $\VQ_m(\Vv_{H_c,h_c}):=(Q_m(v_{H_c}), Q_{m,1}(v_{h_c, 1}), Q_{m,2}(v_{h_c, 2}))$ as 
\begin{equation*}
\VQ_m(\Vv_{H_c, h_c}):=\sum_{(T, S_1, S_2)\in \CT_{H_c}\times \CT_{h_c}(Y^*)\times \CT_{h_c}(D)}\VQ_{T\times S, m}(\Vv_{H_c, h_c}|_{T\times S}),
\end{equation*}
where $\VQ_{T\times S, m}(\Vv_{H_c, h_c}|_{T\times S})\in \VW_{H,h}(G_T\times Y_S)$ solves the following subscale corrector problem
\begin{align}
\label{eq:subscalecorrec}
\CB_{G_T\times Y_S}(\Vw, \VQ_{T\times S, m}(\Vv_{H_c, h_c}|_{T\times S}))=\CB_{T\times S}(\Vw, \Vv_{H_c, h_c})\quad \forall \Vw\in \VW_{H, h}(G_T\times Y_S).
\end{align}
The space of test functions then reads
\[\overline{\VV}_{H_c, h_c, m}:=(1-\VQ_m)(\VV_{H_c, h_c})\]
and can be written as triple
\[\overline{\VV}_{H_c, h_c, m}=\overline{V}_{H_c,m}\times L^2(\Om;\overline{V}_{h_c, m}(Y^*))\times L^2(\Om; \overline{V}_{h_c, m}(D)).\]
We emphasize that $\dim\overline{\VV}_{H_c, h_c,m}=\dim \VV_{H_c, h_c}$ is low-dimensional and the dimension does not depend on $H$, $h$, or $m$.

\begin{definition}
\label{def:LODlocal}
The two-scale Localized Orthogonal Decomposition in Petrov-Galerkin formulation seeks $\Vu_{H_c, h_c}\in \VV_{H_c, h_c}$ such that
\begin{equation}
\label{eq:LODlocal}
\CB(\Vu_{H_c, h_c},\overline{\Vpsi}_{H_c, h_c})=(g, \overline{\psi}_{H_c})_{\partial G}\qquad \forall \overline{\Vpsi}_{H_c, h_c}:=(\overline{\psi}_{H_c}, \overline{\psi}_{h_c, 1}, \overline{\psi}_{h_c, 2})\in \overline{\VV}_{H_c, h_c, m}. 
\end{equation}
\end{definition}
The error analysis will show that the choice $k(H_c+h_c)\lesssim 1$ and $m\approx \log k$ suffices to guarantee stability and quasi-optimality of the method, provided that the direct discretization \eqref{eq:discrtwoscaleeq} (with mesh widths $H$, $h$) is stable.

As discussed in \cite{P15LODhelmholtz}, further stable variants of the method are possible: The local subscale correction procedure can be applied to only the test functions, only the ansatz functions, or both ansatz and test functions.

\subsection{Remarks on implementation aspects}
The present approach of the LOD exploits the two-scale structure of the underlying problem. In practice, one cannot work with the space triples such as $\VV_{H_c, h_c}$, but will look at each of the function spaces separately.
The LOD consists of two main steps: First, the modified basis functions in $\overline{\VV}_{H_c, h_c, m}$ have to be determined, which includes the solution of the localized subscale corrector problems \eqref{eq:subscalecorrec}.  Second, the actual LOD-approximation is computed as solution to \eqref{eq:LODlocal}.
In this section, we explain how the computations in the macroscopic domain $G$ and on the unit square $Y$ can be decoupled in both steps.
For general considerations on how to implement an LOD, for example algebraic realizations of the problems, we refer to \cite{EHMP16LODimplementation}.

\textbf{Computation of modified bases.}\hspace{1ex}
We observe that due to the sesquilinearity of $\CB$ the following linearity for the correction operators $\VQ_m$ holds
\begin{align*}
\VQ_m\Vv_{H_c, h_c}&=\VQ_m(v_{H_c}, 0, 0)+\VQ_m(0, v_{h_c, 1}, 0)+\VQ_m(0, 0, v_{h_c, 2})\\*
&=(Q_m(v_{H_c}), 0, 0)+(0, Q_{m,1}(v_{h_c}), 0)+(0,0, Q_{m, 2}(v_{h_c, 2})).
\end{align*} 
This means that the corrections of the basis functions in $V_{H_c}^1$, $\widetilde{V}_{h_c}^1(Y^*)$ and $V_{h_c}^1(D)$ can be computed separately in the following way:
\begin{enumerate}
\item Choose a basis $\{\lambda_x\}$ of $V_{H_c}^1$, $\{\lambda_{y,1}\}$ of $\widetilde{V}_{h_c}^1(Y^*)$ and $\{\lambda_{y,2}\}$ of $V_{h_c}^1(D)$.
\item For each basis function $\lambda_x$, $\lambda_{y,1}$ and $\lambda_{y,2}$ do
\begin{enumerate}
\item Find the solutions $Q_{T\times S,m}(\lambda_x)$, $Q_{T\times S, m, 1}(\lambda_{y,1})$ and $Q_{T\times S, m, 2}(\lambda_{y,2})$ of the corrector problem \eqref{eq:subscalecorrec} for each $T\in \CT_{H_c}$, $S_1\in \CT_{h_c}(Y^*)$ and $S_2\in \CT_{h_c}(D)$.
This needs the determination of $W_H(G_T)$, $W_h(Y^*_S)$ and $W_h(D_S)$.
\item Build up the modified bases $\overline{\lambda}_x$ of $\overline{V}_{H_c, m}$, $\overline{\lambda}_{y,1}$ of $\overline{V}_{h_c,m}(Y^*)$ and $\overline{\lambda}_{y,2}$ of $\overline{V}_{h_c, m}(D)$ via $\overline{\lambda}_x:=\lambda_x-\sum_{(T, S_1, S_2)\in \CT_{H_c}\times \CT_{h_c}(Y^*)\times \CT_{h_c}(D)}Q_{T\times S,m}(\lambda_x)$, etc.
\end{enumerate}
\end{enumerate}
Note that no communication between the basis functions on $G$, $Y^*$, and $D$ is needed and therefore, the computation of the modified bases can be easily parallelized. 
Only if the parameters $\ep_i$ and $\ep_e$ are constant w.r.t.\ $x$ as here, the corrections $Q_{T\times S, m, 1}$ and $Q_{T\times S, m, 2}$ are $x$-independent.
Depending on the choice of the interpolation operator, Lagrange multipliers can be employed to decode that a function belongs to $W_H$ or $W_h$, see \cite{EHMP16LODimplementation}.

\smallskip
We can further decrease the computational complexity of the localized corrector problems by decoupling the integrals over $G$ and $Y$ and by reducing the number of correction problems.
The potential gain of course hinges on (additional) structure of the parameters and the meshes with the following general observations:
\begin{itemize}
\item The corrections $Q_{T\times S, m,1} $ and $Q_{T\times S, m, 2}$ only have to be computed for $T\in \CT_{H_c}$ with $T\cap \Om\neq \emptyset$.
\item It is sufficient to choose test functions of the form $\Vw=(w, 0, 0)$ for $Q_{T\times S, m}$, $\Vw=(0, w_1, 0)$ for $Q_{T\times S, m,1 }$, and $\Vw=(0, 0, w_2)$ for $Q_{T\times S, m, 2}$. 
\item In the case of constant parameters $\ep_e$ and $\ep_i$, the corrector problems for $Q_{T\times S, m, 1}$ and $Q_{T\times S, m, 2}$ include information on $T$ only in form of the weights $|T|$ and $|\Om_T|$; and the problems for $Q_{T\times S, m}$ only depend on $S$ in form of the weights  $|S_1|$ and $|Y^*_S|$.
\item In case of structured meshes $\CT_{H_c}$ and $\CT_{h_c}$ and constant parameters, we can exploit symmetries to reduce the number of corrector problems \cite{GP15scatteringPG}.
\end{itemize}

\textbf{Computation of the LOD-approximation.}\hspace{1ex} 
The LOD-approximation is defined as the solution to \eqref{eq:LODlocal}. 
This problem is similar to the discrete two-scale equation \eqref{eq:discrtwoscaleeq}, only the test functions have been modified. 
Therefore, the LOD-approximation can be re-interpreted as an HMM-approximation with modified test functions and corrector problems.
To be more explicit, $\Vu_{H_c, h_c}\in \VV_{H_c, h_c}$ from Definition \ref{def:LODlocal} can be characterized as $\Vu_{H_c, h_c}=(u_{H_c}, K_{h_c, 1}(u_{H_c}), K_{h_c, 2}(u_{H_c}))$, where $u_{H_c}\in V_{H_c}^1$ is the solution to a HMM with modified test functions and the corrections $K_{h_c, 1}(u_{H_c})$ and $K_{h_c, 2}(u_{H_c})$ are computed from $u_{H_c}$ and its reconstructions as described in \cite{OV16hmmlod1}; see also \cite{Ohl05HMM,HOV15hmmmaxwell} for similar reformulations in different settings.
The HMM with modified test functions involves the following two steps: 
\begin{enumerate}
\item Solve the cell problems for the reconstructions $R_1$ and $R_2$ around each quadrature point of the macroscopic triangulation $\CT_{H_c}$ using  test functions in $\overline{V}_{h_c, m}(Y^*)$ and $\overline{V}_{h_c, m}(D)$.
\item Assemble the macroscopic sesquilinear form $B_H$ with the computed reconstructions and the test functions in $\overline{V}_{H_c, m}$.
\end{enumerate}
Note that the reconstructions $R_1$ and $R_2$ as well as the fine-scale correctors $K_{h,1}$ and $K_{h,2}$ are different from those in \cite{OV16hmmlod1} because of the modified test functions.
This reformulation of the LOD-approximation as solution to a (modified) HMM decouples  the computations on $Y$ and $G$ and no function triple spaces have to be considered. 
This is one great advantage of the present Petrov-Galerkin ansatz for the LOD in comparison to a Galerkin ansatz:
We only need to compute reconstructions of standard Lagrange basis functions in $V_{H_c}^1$, but not of the basis functions in $\overline{V}_{H_c, m}$.

\section{Error analysis}
\label{sec:error}
The error analysis is based on the observation that the localized subscale corrector problems \eqref{eq:subscalecorrec} can be seen as perturbation of idealized subscale problems posed on the whole domain $G\times Y$.
So let us introduce idealized counterparts of the correction operators $\VQ_{T\times S, m}$ and $\VQ_m$ where the patch $G_T\times Y_S$ equals $G\times Y$, roughly speaking ``$m=\infty$''. 
Define the null space
\begin{equation*}
\begin{split}
\VW_{H,h}&:=W_H\!\times\! L^2(\Om; W_h(Y^*))\!\times\! L^2(\Om; W_h(D)):=\{\Vv_{H,h}\in \VV_{H,h}|\VI_{H_c, h_c}(\Vv_{H,h})=0\}.
\end{split}
\end{equation*}
For any $\Vv_{H,h}\in \VV_{H,h}$, the idealized element corrector problem seeks $\VQ_{T\times S, \infty}\Vv_{H,h}\in \VW_{H,h}$ such that
\begin{equation}
\label{eq:idealsubscalecorrec}
\CB(\Vw, \VQ_{T\times S, \infty}\Vv_{H,h})=\CB_{T\times S}(\Vw, \Vv_{H,h})\qquad \forall\Vw\in \VW_{H,h},
\end{equation} 
and we define
\begin{equation}
\label{eq:Qinfty}
\VQ_{\infty}(\Vv_{H,h}):=\sum_{(T, S_1, S_2)\in \CT_{H_c}\times \CT_{h_c}(Y^*)\times \CT_{h_c}(D)}\VQ_{T\times S, \infty}(\Vv_{H,h}).
\end{equation}
The following result implies the well-posedness of the idealized corrector problems.

\begin{lemma}
\label{lem:wellposedcorrec}
Under the assumption
\begin{equation}
\label{eq:resolassptcorrec}
k(C_{I_{H_c}}\sqrt{C_{ol, G}}\,H_c+C_{I_{h_c}^D}\sqrt{C_{ol, Y}}\,h_c)\leq \sqrt{C_{\min}/2},
\end{equation}
we have for all $\Vw_{h,h}:=(w_H, w_{h,1}, w_{h,2})\in \VW_{H,h}$ the following equivalence of norms
\[\|(w_H, w_{h,1}, w_{h,2})\|_{1,e}\leq \|(w_H, w_{h,2}, w_{h,2})\|_e\leq \sqrt{1+C_{\min}/2}\;\|(w_H, w_{h,1}, w_{h,2})\|_{1,e},\]
and coercivity
\[C_{\min}/2\; \|(w_H, w_{h,1}, w_{h,2})\|^2_{1,e}\leq \Re\CB(\Vw_{H,h}, \Vw_{H,h}),\]
where the $H^1$-semi norm $\|\cdot\|_{1,e}$ is defined in \eqref{eq:errorh1semi}.
\end{lemma}
\begin{proof}
The essential observation is that for any $(w_h, w_{h,1}, w_{h,2})\in \VW_{H,h}$ the property of the quasi-interpolation operators \eqref{eq:intpol} implies that
\begin{align*}
&\!\!\!\!k^2\|w_H+\chi_Dw_{h,2}\|^2_{G\times Y}\\
&\leq k^2\bigl(\|w_H\|_G+\|w_{h,2}\|_{G\times D}\bigr)^2\\
&= k^2\bigl(\|w_H-I_{H_c}(w_H)\|_G+\|w_{h,2}-I_{h_c}^D(w_{h,2})\|_{G\times D}\bigr)^2\\
&\leq k^2\bigl(H_c\, C_{I_{H_c}}\sqrt{C_{ol, G}}\, \|\nabla w_H\|_G+h_c\, C_{I_{h_c}^D}\sqrt{C_{ol, Y}}\|\nabla_y w_{h,2}\|_{G\times D}\bigr)^2.
\end{align*}
This directly yields the equivalence of norms on $\VW_{H,h}$ under the resolution condition \eqref{eq:resolassptcorrec}.  For the coercivity we observe that
\begin{equation*}
\Re\CB(\Vw_{H,h}, \Vw_{H,h})\geq C_{\min}\|\Vw_{H,h}\|^2_{1,e}-k^2\|w_H+\chi_D w_{h,2}\|^2_{G\times Y}.
\end{equation*}
\end{proof}

As the sesquilinear form $\CB$ is also continuous (see Lemma \ref{lem:propertiesCB}), Lemma \ref{lem:wellposedcorrec} implies that the idealized corrector problem \eqref{eq:idealsubscalecorrec} is well-posed and that the idealized correctors $\VQ_\infty$ defined by \eqref{eq:Qinfty} are continuous w.r.t.\ the two-scale energy norm
\[\|\VQ_\infty(\Vv_{H,h})\|_e\leq C_{\VQ}\,\|\Vv_{H,h}\|_e \qquad \text{for all}\quad \Vv_{H,h}\in \VV_{H,h}.\]
Since the inclusion $\VW_{H,h}(G_T\times Y_S)\subset \VW_{H,h}$ holds, the well-posedness result carries over to the localized corrector problems \eqref{eq:subscalecorrec} with the same constant.

The proof of the well-posedness of the two-scale LOD in Petrov-Galerkin formulation \eqref{eq:LODlocal} relies on the fact that $(\VQ_\infty-\VQ_m)(\Vv)$ decays exponentially with the distance from $\supp(\Vv)$. The difference between idealized and localized correctors is quantified in the next theorem. The proof is given in Section \ref{sec:decay} and is based on the observation that $\VQ_\infty(\Vv|_{T\times S})$ decays exponentially with distance from $T\times S$.

\begin{theorem}
\label{thm:decaycorrec}
Under the resolution condition \eqref{eq:resolassptcorrec} there exist constants $C_1$, $C_2$, and $0<\be<1$, independent of $H_c$, $h_c$, $H$, and $h$,  such that for any $\Vv_{H_c,h_c}\in \VV_{H_c, h_c}$, any $(T, S_1, S_2)\in \CT_{H_c}\times\CT_{h_c}(Y^*)\times\CT_{h_c}(D)$ and any $m\in \nz$ it holds
\begin{align}
\label{eq:decaydiffcorrecT}
\|(\VQ_{T\times S, \infty}-\VQ_{T\times S, m})(\Vv_{H_c, h_c})\|_{1,e}&\leq C_1\be^m\|\Vv_{H_c, h_c}\|_{1,e,T\times S},\\
\label{eq:decaydiffcorrec}
\|(\VQ_\infty-\VQ_m)(\Vv_{H_c, h_c})\|_{1,e}&\leq C_2(\sqrt{C_{ol,m,G}}+\sqrt{C_{ol,m, Y}})\be^m\|\Vv_{H_c, h_c}\|_{1,e}.
\end{align}
\end{theorem}

The stability of the LOD requires the coupling of the oversampling parameter $m$ to the stability-/inf-sup-constant of the HMM. 
Therefore, we assume that $H$ and $h$ are small enough that \eqref{eq:stabHMM} holds.

\begin{theorem}[Well-posedness of the LOD]
\label{thm:stabilLOD}
Under the resolution conditions \eqref{eq:resolassptcorrec} and \eqref{eq:stabHMM} and the following oversampling condition
\begin{equation}
\label{eq:oversamplingstab}
m\geq \frac{(q\!+\!1)\log(k)\!+\!\log\bigl(2C_2C_{\VI}C_{\VI,e}C_{\HMM}C_B\sqrt{1+ C_{\min}/2}\bigl(\sqrt{C_{\ol, m, G}}\!+\!\sqrt{C_{\ol, m, Y}}\bigr)\bigr)}{|\log(\be)|},
\end{equation}
the two-scale LOD \eqref{eq:LODlocal} is well-posed and with the constant $C_{\LOD}:=2 C_{\HMM}C_{\VI, e}^2(1+C_{\VQ})$ it holds
\[ (C_{\LOD}\,k^{q+1})^{-1}\leq \inf_{\Vv_{H_c, h_c}\in \VV_{H_c, h_c}}\:\sup_{\overline{\Vpsi}_{H_c, h_c}\in \overline{\VV}_{H_c, h_c, m}}\frac{\Re \CB(\Vv_{H_c, h_c}, \overline{\Vpsi}_{H_c, h_c})}{\|\Vv_{H_c, h_c}\|_e\|\overline{\Vpsi}_{H_c, h_c}\|_e}.\]
\end{theorem}

As $C_{\ol, m, G}$ and $C_{\ol, m, Y}$ grow at most polynomially with $m$ because of the quasi-uniformity of $\CT_{H_c}$ and $\CT_{h_c}$, condition \eqref{eq:oversamplingstab} is indeed satisfiable and the choice of the oversampling parameter $m$ will be dominated by the logarithm of the wave number.

\begin{proof}
Let $\Vv_{H_c, h_c}\in \VV_{H_c, h_c}$ be given. From \eqref{eq:stabHMM} we infer that there is $\Vpsi\in \VV_{H,h}$ such that
\[\Re \CB(\Vv_{H_c,h_c}-(\VQ_\infty(\Vv_{H_c,h_c}^*))^*, \Vpsi)\geq (C_{\HMM}^{-1}k^{-(q+1)})\|\Vv_{H_c,h_c}-(\VQ_\infty(\Vv_{H_c,h_c}^*))^*\|_e\, \|\Vpsi\|_e.\]
It follows from the structure of the sesquilinear form $\CB$ that $(\VQ_\infty(\Vv_{H_c,h_c}^*))^*$ solves the following adjoint corrector problem 
\[\CB((\VQ_\infty(\Vv_{H_c,h_c}^*))^*, \Vw)=\CB(\Vv_{H_c,h_c}, \Vw)\qquad \forall \Vw\in \VW_{H,h}.\]
Let $\overline{\Vpsi}_{H_c, h_c}:=(1-\VQ_m)\VI_{H_c, h_c}\Vpsi\in \overline{\VV}_{H_c, h_c, m}$. It obviously holds that
\begin{equation}
\label{eq:CBforLOD}
\CB(\Vv_{H_c, h_c}, \overline{\Vpsi}_{H_c, h_c})=\CB(\Vv_{H_c, h_c}, (1-\VQ_\infty)\VI_{H_c, h_c}\Vpsi)
+\CB(\Vv_{H_c, h_c}, (\VQ_\infty-\VQ_m)\VI_{H_c, h_c}\Vpsi).
\end{equation}
Since $\VQ_\infty$ is a projection onto $\VW_{H, h}$ and $(1-\VI_{H_c, h_c})\Vpsi\in \VW_{H,h}$, we have $(1-\VQ_\infty)(1-\VI_{H_c, h_c})\Vpsi=0$ and thus, $(1-\VQ_\infty)\VI_{H_c, h_c}\Vpsi=(1-\VQ_\infty)\Vpsi$. The solution property of $(\VQ_\infty(\Vv_{H_c, h_c}^*))^*$ and the definition of $\VQ_\infty$ in \eqref{eq:idealsubscalecorrec}--\eqref{eq:Qinfty} gives
\begin{equation*}
\begin{split}
\CB((\VQ_\infty(\Vv_{H_c, h_c}^*))^*, \Vpsi)
&=\CB((\VQ_\infty(\Vv_{H_c, h_c}^*))^*, \VQ_\infty\Vpsi)+\CB((\VQ_\infty(\Vv_{H_c, h_c}^*))^*, (1-\VQ_\infty)\Vpsi)\\
&=\CB(\Vv_{H_c, h_c}, \VQ_\infty \Vpsi).
\end{split}
\end{equation*}
Hence, we obtain
\begin{align*}
\Re\CB(\Vv_{H_c, h_c}, (1-\VQ_\infty)\VI_{H_c, h_c}\Vpsi)&=\Re\CB(\Vv_{H_c, h_c}-(\VQ_\infty(\Vv_{H_c, h_c}^*))^*, \Vpsi)\\
&\geq (C_{\HMM}\,k^{q+1})^{-1}\,\|\Vv_{H_c, h_c}-(\VQ_\infty(\Vv_{H_c, h_c}^*))^*\|_e\, \|\Vpsi\|_e.
\end{align*}
Furthermore, the estimate \eqref{eq:intpolenergystable} implies
\[\|\Vv_{H_c, h_c}\|_e=\|\VI_{H_c, h_c}(\Vv_{H_c, h_c}-(\VQ_\infty(\Vv_{H_c, h_c}^*))^*)\|_e\leq C_{\VI, e}\|\Vv_{H_c, h_c}-(\VQ_\infty(\Vv_{H_c, h_c}^*))^*\|_e\]
and
\[\|\overline{\Vpsi}_{H_c, h_c}\|_e\leq C_{\VI, e}(1+C_{\VQ})\,\|\Vpsi\|_e.\]
The second term on the right-hand side of \eqref{eq:CBforLOD} satisfies with Lemma \ref{lem:wellposedcorrec} and Theorem \ref{thm:decaycorrec} that
\begin{align*}
&\!\!\!\!|\CB(\Vv_{H_c, h_c}, (\VQ_\infty-\VQ_m)\VI_{H_c, h_c}\Vpsi)|\\*
&\leq \sqrt{1+ C_{\min}/2}\;C_B\,\|(\VQ_\infty-\VQ_m)\VI_{H_c, h_c}\Vpsi\|_{1,e}\, \|\Vv_{H_c, h_c}\|_e\\
&\leq \sqrt{1+C_{\min}/2}\; C_BC_2 \bigl(\sqrt{C_{\ol,m,G}}+\sqrt{C_{\ol,m,Y}}\bigr) \be^m C_{\VI}\, \|\Vpsi\|_e\, \|\Vv_{H_c, h_c}\|_e.
\end{align*}
Altogether, this  yields
\begin{align*}
&\!\!\!\!\Re\CB(\Vv_{H_c, h_c}, \overline{\Vpsi}_{H_c, h_c})\\*
&\geq \Bigl(\frac{1}{C_{\VI, e}C_{\HMM}k^4}-\sqrt{1+C_{\min}/2}\;C_BC_2 \bigl(\sqrt{C_{ol,m,G}}+\sqrt{C_{ol,m,Y}}\bigr) \be^mC_{\VI} \Bigr)\\*
&\qquad\cdot\|\Vv_{H_c, h_c}\|_e\,\|\Vpsi\|_e\\
&\geq \Bigl(\frac{1}{C_{\VI, e}C_{\HMM}k^4}-\sqrt{1+C_{\min}/2}\;C_BC_2 \bigl(\sqrt{C_{ol,m,G}}+\sqrt{C_{ol,m,Y}}\bigr) \be^mC_{\VI} \Bigr)\\*
&\qquad \cdot\frac{1}{C_{\VI, e}(1+C_{\VQ})}\, \|\Vv_{H_c, h_c}\|_e\, \|\overline{\Vpsi}_{H_c, h_c}\|_e.
\end{align*}
Hence, the condition \eqref{eq:oversamplingstab} implies the assertion.
\end{proof}

\begin{remark}[Adjoint problem]
\label{rem:adjointLOD}
Under the assumption of Theorem \ref{thm:stabilLOD}, problem \eqref{eq:LODlocal} is well-posed. Thus, it follows from a dimension argument that also the adjoint problem to \eqref{eq:LODlocal} is well-posed with the same stability constant as in Theorem \ref{thm:stabilLOD}, cf.\ \cite[Remark 1]{GP15scatteringPG}.
\end{remark}

\begin{theorem}[Quasi-optimality]
\label{thm:quasioptLOD}
Under the resolution conditions \eqref{eq:resolassptcorrec} and \eqref{eq:stabHMM} and the oversampling conditions \eqref{eq:oversamplingstab} and 
\begin{equation}
\label{eq:oversamplingquaisopt}
m\geq \bigl((q+1)\log(k)+\log(2\sqrt{1+C_{\min}/2}\; C_B^2 C_2C_{\LOD})\bigr)/|\log(\be)|,
\end{equation} 
the LOD-approximation $\Vu_{H_c, h_c}$, solution to \eqref{eq:LODlocal},  and the solution $\Vu_{H,h}$ of the direct discretization \eqref{eq:discrtwoscaleeq} satisfy
\begin{equation*}
\|\Vu_{H, h}-\Vu_{H_c, h_c}\|_e\leq C\min_{\Vv_{H_c, h_c}\in \VV_{H_c, h_c}}\|\Vu_{H, h}-\Vv_{H_c, h_c}\|_e
\end{equation*}
with a generic constant $C$ depending only on $C_{\VI, e}$.
\end{theorem}

\begin{proof}
Let $\Ve:=\Vu_{H,h}-\Vu_{H_c, h_c}$. We prove that $\|\Ve\|_e\leq 2\|(1-\VI_{H_c, h_c})\Vu_{H,h}\|_e$, which gives the assertion because $\VI_{H_c, h_c}$ is a projection.
By the triangle inequality and the fact that $\VI_{H_c, h_c}$ is a projection onto $\VV_{H_c, h_c}$, we obtain
\[\|\Ve\|_e\leq \|(1-\VI_{H_c, h_c})\Vu_{H, h}\|_e+\|\VI_{H_c, h_c}\Ve\|_e,\]
so that it only remains to bound the second term on the right-hand side.
The proof employs a standard duality argument, the stability of the idealized method and the fact that the actual two-scale LOD can be seen as a perturbation of the idealized method.
Let $\Vz_{H_c, h_c}\in \VV_{H_c, h_c}$ be the solution to the dual problem
\begin{align*}
\CB(\Vpsi_{H_c, h_c}, (1-\VQ_\infty)\Vz_{H_c, h_c})&=(\Vpsi_{H_c, h_c}, \VI_{H_c, h_c}\Ve)_e \qquad \forall \Vpsi_{H_c, h_c}\in \VV_{H_c, h_c},
\end{align*}
where $(\cdot, \cdot)_e$ denotes the scalar product which induces the two-scale energy norm \eqref{eq:errornorm}. This adjoint problem is uniquely solvable as explained in Remark \ref{rem:adjointLOD}.
Choosing the test function $\Vpsi_{H_c, h_c}=\VI_{H_c, h_c}\Ve$ implies
\begin{equation}
\label{eq:Pie}
\begin{split}
\|\VI_{H_c, h_c}\Ve\|_e^2&=\CB(\VI_{H_c, h_c}\Ve, (1-\VQ_\infty)\Vz_{H_c, h_c})\\
&=\CB(\VI_{H_c, h_c}\Ve, (\VQ_m-\VQ_\infty)\Vz_{H_c, h_c})+\CB(\VI_{H_c, h_c}\Ve, (1-\VQ_m)\Vz_{H_c, h_c}).
\end{split}
\end{equation}
Since $(1-\VQ_m)\Vz_{H_c, h_c}\in \overline{\VV}_{H_c, h_c, m}$ by definition, we have the Galerkin orthogonality
\[\CB(\Vu_{H,h}-\Vu_{H_c, h_c}, (1-\VQ_m)\Vz_{H_c, h_c})=0.\]
Using this orthogonality and the fact that $\VI_{H_c, h_c}\Vu_{H, h}-\Vu_{H, h}\in \VW_{H,h}$ together with the definition of $\VQ_\infty$ \eqref{eq:idealsubscalecorrec}--\eqref{eq:Qinfty} implies for the second term
\begin{align*}
\CB(\VI_{H_c, h_c}\Ve, (1-\VQ_m)\Vz_{H_c, h_c})&=\CB(\VI_{H_c, h_c}\Vu_{H,h}-\Vu_{H,h}, (1-\VQ_m)\Vz_{H_c, h_c})\\
&=\CB(\VI_{H_c, h_c}\Vu_{H,h}-\Vu_{H,h}, (\VQ_\infty-\VQ_m)\Vz_{H_c, h_c}).
\end{align*}
Now the first and the (modified) second term of \eqref{eq:Pie} are similar and can be treated with the same procedure. First, we note that $(\VQ_\infty-\VQ_m)\Vz_{H_c, h_c}\in \VW_{H,h}$. Applying Lemma \ref{lem:wellposedcorrec} and then the decay estimate \eqref{eq:decaydiffcorrec} from Theorem \ref{thm:decaycorrec}, we obtain (for the second term)
\begin{align*}
&\!\!\!\!|\CB(\VI_{H_c, h_c}\Vu_{H, h}-\Vu_{H,h}, (\VQ_\infty-\VQ_m)\Vz_{H_c, h_c})|\\*
&\leq \sqrt{1+C_{\min}/2}\;C_B\|(1-\VI_{H_c, h_c})\Vu_{H,h}\|_e\, \|(\VQ_\infty-\VQ_m)\Vz_{H_c, h_c}\|_{1,e}\\*
&\leq \sqrt{1+C_{\min}/2}\;C_BC_2 \bigl(\sqrt{C_{ol,m,G}}+\sqrt{C_{ol,m,Y}}\bigr) \be^m\, \|(1-\VI_{H_c, h_c})\Vu_{H,h}\|_e\, \|\Vz_{H_c, h_c}\|_{1,e}.
\end{align*}
The stability of the adjoint problem from Remark \ref{rem:adjointLOD} implies
\[\|\Vz_{H_c, h_c}\|_{1,e}\leq C_{\LOD}k^{q+1}C_B\|\VI_{H_c, h_c}\Ve\|_e.\]
Thus, we obtain for \eqref{eq:Pie} after division by $\|\VI_{H_c,h_c}\Ve\|_e$ that
\begin{align*}
\|\VI_{H_c, h_c}\Ve\|_e&\leq \sqrt{1+C_{\min}/2}\;C_B^2C_2 \bigl(\sqrt{C_{ol,m,G}}+\sqrt{C_{ol,m,Y}}\bigr) \be^m C_{\LOD}k^{q+1}\\
&\qquad \cdot(\|(1-\VI_{H_c, h_c})\Vu_{H,h}\|_e+\|\VI_{H_c, h_c}\Ve\|_e).
\end{align*}
The oversampling condition \eqref{eq:oversamplingquaisopt} implies that the constants can be bounded by $1/2$ and hence, the term $\|\VI_{H_c, h_c}\Ve\|_e$ can be absorbed on the left-hand side. 
\end{proof}

\begin{corollary}[Full error]
Let $\Vu:=(u, u_1, u_2)\in \CH$ be the two-scale solution to \eqref{eq:twoscaleeq}. Under the assumptions of Theorem \ref{thm:quasioptLOD}, the two-scale LOD-approximation $\Vu_{H_c, h_c}$, solution to \eqref{eq:LODlocal}, satisfies with some generic constant $C$
\[\|\Vu-\Vu_{H_c, h_c}\|_e\leq \|\Vu-\Vu_{H,h}\|_e+C\min_{\Vv_{H_c, h_c}\in \VV_{H_c, h_c}}\|\Vu_{H, h}-\Vu_{H_c, h_c}\|_e.\]
\end{corollary}
If the error for the direct (HMM-)approximation is small (which is the case for sufficiently small $H$, $h$), the error is dominated by the best approximation error of $\VV_{H_c, h_c}$, which can be quantified using standard interpolation operators and regularity results.

\section{Proof of the decay property for the correctors}
\label{sec:decay}
In this section, we give a proof of the exponential decay result of Theorem \ref{thm:decaycorrec}, which is central for this method. 
The idea of the proof is the same as in the previous proofs for the Helmholtz equation \cite{BGP15hethelmholtzLOD,GP15scatteringPG,P15LODhelmholtz} or in the context of diffusion problems \cite{HMP14partunity,MP14LODelliptic}. 
As in the previous sections, we have to take into account the two-scale nature of the problem and the spaces.

Let $\VCI_{H,h}:=(\CI_{H}, \CI_h^{Y^*}, \CI_h^D)$ with $\CI_{H}: C^0(G)\to V_H^1$, $\CI_{h}^{Y^*}: L^2(\Om; C^0(Y^*))\to L^2(\Om; \widetilde{V}_h(Y^*))$, and $\CI_h^D: L^2(\Om; C^0(D))\to L^2(\Om; V_h^1(D))$ denote the nodal Lagrange interpolation operators, where $\CI_h^{Y^*}$ and $\CI_H^D$ only act on the second variable.
We note that periodicity is preserved when identifying degrees of freedom on the periodic boundary.
Recall that the nodal Lagrange interpolation operator $\CI$ is $L^2$- and $H^1$-stable on piece-wise polynomials on shape regular meshes due to inverse inequalities.
Hence, for any $(T, S_1, S_2)\in \CT_{H_c}\times \CT_{h_c}(Y^*)\times \CT_{h_c}(D)$ and all $\Vq\in \pz_2(T)\times L^2(T; \pz_2(S_1))\times L^2(T; \pz_2(S_2))$ we have the stability estimate
\begin{equation}
\label{eq:nodalintpolh1stable}
\|\VCI_{H,h}\Vq\|_{1, e, T\times S}\leq C_{\VCI}\,\|\Vq\|_{1,e, T\times S}.
\end{equation} 
In this section, we do not explicitly give the constants in the estimates. Instead we use a generic constant $C$, which is independent of the mesh sizes and the oversampling parameter, but may depend on the (quasi)-interpolation operators' norms, the overlap constants $C_{\ol, G}$ and $C_{\ol, Y}$ (not on $C_{\ol, m, G}$ and $C_{\ol, m, Y}$!), the constant for the cut-off functions (see below), and $C_{\min}$.

In the proofs, we will frequently use cut-off functions. We collect some basic properties in the following lemma, cf.\ also \cite[Appendix A, Lemma 2]{GP15scatteringPG}.
\begin{lemma}
Let $\Veta:=(\eta_0, \eta_1, \eta_2)\in \VV_{H_c, h_c}$ be a function triple with $\eta_i$, $i=1,2,3$, having values in the interval $[0,1]$ and satisfying the bounds
\begin{equation}
\label{eq:gradeta}
\|\nabla \eta_0\|_{L^\infty(G)}\leq C_{\eta}H_c^{-1}\qquad \text{and}\qquad \|\nabla_y \eta_i\|_{L^2(\Om, L^\infty(Y))}\leq C_{\eta}h_c^{-1} \quad i=1,2.
\end{equation}
By writing $\nabla \Veta$, we mean the function triple $(\nabla \eta_0, \nabla_y \eta_1, \nabla_y\eta_2)$.
Let $\Vw:=(w, w_1, w_2)\in \VW_{H,h}$ be arbitrary and define $\Veta\Vw:=(\eta_0 w, \eta_1 w_1, \eta_2 w_2)$.
Given any subset $(\CK_0, \CK_1, \CK_2)\subset \CT_{H_c}\times \CT_{h_c}(Y^*)\times \CT_{h_c}(D)$, $\Vw$ fulfills for $\CS_i=\cup\CK_i$ and $\CS:=\CS_0\times (\CS_1\cup\CS_2)$ that
\begin{align}
\label{eq:L2gradW}
\|\Vw\|_{\CS}&\leq C( H_c\|\nabla w\|_{\UN(\CS_0)}\!+\sum_{i=1}^2h_c\|\nabla_y w_i\|_{\CS_0\times \UN(\CS_i)}),\\
\label{eq:nablaeta}
\|\Veta \Vw\|_{1,e, \CS}&\leq C (\| \Vw\|_{1, e, \CS\cap \supp(\Veta)}+\|\Vw\|_{1, e,\UN(\CS\cap \supp(\nabla \Veta))}),\\
\label{eq:L2IHcIh}
\|(1-\VI_{H_c, h_c})\VCI_{H,h}(\Veta \Vw)\|_{\CS}&\leq C(H_c+h_c)\|\Veta\Vw\|_{1, e, \UN(\CS)}.
\end{align}
\end{lemma}
\begin{proof}
The properties \eqref{eq:intpol} directly imply \eqref{eq:L2gradW}. For the proof of \eqref{eq:nablaeta} the product rule  and \eqref{eq:gradeta} yield
\begin{align*}
\|\Veta\Vw\|_{1,e , \CS}&\leq \|\Vw\|_{1,e, \CS\cap\supp(\Veta)}+C_\eta H_c^{-1}\|w\|_{\CS_0\cap\supp(\nabla \eta_0)}\\*
&\quad+\sum_{i=1}^2C_\eta h_c^{-1}\|w_i\|_{\CS_0\times (\CS_i\cap\supp(\nabla \eta_i))}.
\end{align*}
The combination with \eqref{eq:L2gradW} gives the assertion. For a proof of \eqref{eq:L2IHcIh}, apply \eqref{eq:intpol}. The estimate then follows from the $H^1$-stability of $\VCI_{H,h}$ \eqref{eq:nodalintpolh1stable} on the piece-wise polynomial function $\Veta\Vw$. 
\end{proof}

\begin{proposition}
\label{prop:decayQinfty}
Under the resolution condition \eqref{eq:resolassptcorrec}, there exists $0<\be<1$ such that, for any $\Vv_{H_c, h_c}\in \VV_{H_c, h_c}$ and all $(T, S_1, S_2)\in \CT_{H_c}\times \CT_{h_c}(Y^*)\times \CT_{h_c}(D)$ and $m\in\nz$
\begin{align*}
\|\VQ_{T\times S, \infty}\Vv_{H_c, h_c}\|_{1,e, (G\times Y)\setminus \UN^m(T\times S)}\leq C\be^m\|\Vv_{H_c, h_c}\|_{1,e,T\times S}.
\end{align*}
\end{proposition}

\begin{proof}
For $m\geq5$, we define the cut-off functions $\eta_0\in V_{H_c}^1$, $\eta_1\in L^2(\Om; \widetilde{V}_{h_c}^1(Y^*))$, $\eta_2\in L^2(\Om;V_{h_c}^1(D))$ via
\begin{align*}
\eta_0 &=0 \quad\text{ in }\UN^{m-3}(T)&\text{ and }&&\eta_0&=1 \quad\text{ in }G\setminus \UN^{m-2}(T),\\
\eta_1 &=0 \quad\text{ in }\UN^{m-3}(T\times S_1)&\text{ and }&&\eta_1&=1 \quad\text{ in }(\Om\times Y^*)\setminus \UN^{m-2}(T\times S_1),\\
\eta_2 &=0 \quad\text{ in }\UN^{m-3}(T\times S_2)&\text{ and }&&\eta_2&=1 \quad\text{ in }(\Om\times D)\setminus \UN^{m-2}(T\times S_2),
\end{align*}
where $\eta_1$ and $\eta_2$ w.l.o.g.\ are chosen piece-wise $x$-constant.
The shape regularity implies that $\eta_i$ satisfies \eqref{eq:gradeta}. Denote $\Veta:=(\eta_0, \eta_1, \eta_2)$ and $\CR:=\supp(\nabla \Veta)$. Let $\Vv_{H_c, h_c}\in \VV_{H_c, h_c}$ and denote $\Vphi:=\VQ_{T\times S, \infty}\Vv_{H_c, h_c}\in \VW_{H,h}$. Elementary estimates yield
\begin{align*}
&\!\!\!\!\!\|\Vphi\|_{1,e,(G\times Y)\setminus \UN^m(T\times S)}^2\\
&= |\Re(\Vphi, \Vphi)_{1,e,(G\times Y)\setminus \UN^m(T\times S)}|\\
&\leq |\Re\bigl((\nabla \phi+\nabla_y\phi_1, \eta_0\nabla\phi+\eta_1\nabla_y\phi_1)_{G\times Y^*}+(\nabla_y \phi_2, \eta_2 \nabla_y\phi_2)_{G\times D}\bigr)|\\
&\leq |\Re(\Vphi, \Veta\Vphi)_{1,e}|+|\Re\bigl((\nabla \phi+\nabla_y\phi_1, \phi\nabla \eta_0+\phi_1\nabla_y\eta_1)_{G\times Y^*}+(\nabla_y\phi_2, \phi_2\nabla_y\eta_2)_{G\times D}\bigr)|\\
&\leq M_1+M_2+M_3+M_4
\end{align*}
for
\begin{align*}
M_1&:=|\Re\bigl(\Vphi, (1-\VCI_{H, h})(\Veta\Vphi)\bigr)_{1,e}|,\\
M_2&:=|\Re\bigl(\Vphi, (1-\VI_{H_c, h_c})\VCI_{H,h}(\Veta\Vphi)\bigr)_{1,e}|,\\
M_3&:=|\Re\bigl(\Vphi, \VI_{H_c, h_c}\VCI_{H,h}(\Veta\Vphi)\bigr)_{1,e}|,\\
M_4&:=|\Re\bigl((\nabla \phi+\nabla_y\phi_1, \phi\nabla \eta_0+\phi_1\nabla_y\eta_1)_{G\times Y^*}+(\nabla_y\phi_2, \phi_2\nabla_y\eta_2)_{G\times D}\bigr)|.
\end{align*}
With the stability of $\VCI_{H, h}$ on polynomials \eqref{eq:nodalintpolh1stable} and estimate \eqref{eq:nablaeta} we obtain
\begin{align*}
M_1&\leq C \|\Vphi\|_{1,e,\CR}\,\|\Veta\Vphi-\VCI_{H,h}(\Veta\Vphi)\|_{1,e,\CR}\leq C\|\Vphi\|_{1,e,\CR}\,\|\Vphi\|_{1,e,\UN(\CR)}.
\end{align*}
Since $\Vw:=(1-\VI_{H_c, h_c})\VCI_{H,h}(\Veta\Vphi)\in \VW_{H,h}$, the idealized corrector problem \eqref{eq:idealsubscalecorrec} and the fact that $\Vw$ has support only outside $T\times S$ imply $\CB(\Vw, \Vphi)=\CB_{T\times S}(\Vw, \Vv_{H_c, h_c})=0$. Therefore, we obtain
\begin{align*}
M_2:=|\Re(\Vphi, \Vw)_{1,e}|&\leq C_{\min}^{-1}\,|\Re\bigl(\CB(\Vw, \Vphi)+k^2(w+\chi_Dw_2, \phi+\chi_D \phi_2)_{G\times Y}\bigr)|\\
&=C_{\min}^{-1}|\Re k^2(w+\chi_Dw_2, \phi+\chi_D \phi_2)_{G\times Y}|.
\end{align*}
Hence, estimates \eqref{eq:L2IHcIh} and \eqref{eq:nablaeta} give with the resolution condition \eqref{eq:resolassptcorrec}
\begin{align*}
M_2&\leq C_{\min}^{-1}k^2(H_c^2C_{I_{H_c}}^2C_{\ol, G}+h_c^2C_{I_{h_c}^D}^2C_{\ol, Y})\|\Vphi\|^2_{1,e, (G\times Y)\setminus \UN^m(T\times S)}\\
&\quad+C_{\min}^{-1}k^2(H_c^2C_{I_{H_c}}^2C_{\ol, G}+h_c^2C_{I_{h_c}^D}^2C_{\ol, Y})C_{\VI}C_{\VCI}C_\eta\|\Vphi\|^2_{1, e, \UN^2(\CR)}\\
&\leq \halb \|\Vphi\|^2_{1,e,(G\times Y)\setminus \UN^m(T\times S)}+C\|\Vphi\|^2_{1,e,\UN^2(\CR)},
\end{align*} 
so that the first term can be absorbed.
Because of $\supp(\VI_{H_c, h_c}\VCI_{H,h}(\Veta\Vphi))\subset \UN(\CR)$, the properties \eqref{eq:intpol} of $\VI_{H_c, h_c}$, \eqref{eq:nodalintpolh1stable} of $\VCI_{H,h}$, and estimate \eqref{eq:nablaeta} lead to 
\begin{align*}
M_3&\leq \|\Vphi\|_{1, e, \UN(\CR)}\, \|\VI_{H_c, h_c}\VCI_{H, h}(\Veta\Vphi)\|_{1,e,\UN(\CR)}\leq C\|\Vphi\|^2_{1, e, \UN^2(\CR)}.
\end{align*}
For the last term, the Lipschitz bound \eqref{eq:gradeta} on the cut-off functions and estimate \eqref{eq:L2gradW} show
\begin{align*}
M_4\leq C_\eta\bigl(H_c^{-1}\|\phi\|_{\supp(\nabla \eta_0)}+\sum_{i=1}^2h_c^{-1}\|\phi_i\|_{\supp(\nabla \eta_i)}\bigr)\|\Vphi\|_{1,e,\CR}\leq C\|\Vphi\|^2_{1, e, \UN(\CR)}.
\end{align*}
All in all, it follows for some $\tilde{C}>0$ that
\[\halb\|\Vphi\|^2_{1,e,(G\times Y)\setminus \UN^m(T\times S)}\leq \tilde{C}\|\Vphi\|^2_{1,e,\UN^2(\CR)},\]
where we recall that $\UN^2(\CR)=\UN^{m}(T\times S)\setminus \UN^{m-5}(T\times S)$.
Because of
\[\|\Vphi\|_{1,e,(G\times Y)\setminus \UN^m(T\times S)}^2+\|\Vphi\|^2_{1,e,\UN^m(T\times S)\setminus \UN^{m-5}(T\times S)}=\|\Vphi\|^2_{1,e,(G\times Y)\setminus \UN^{m-5}(T\times S)},\]
we obtain
\[(1+(2\tilde{C})^{-1})\|\Vphi\|^2_{1,e,(G\times Y)\setminus \UN^m(T\times S)}\leq \|\Vphi\|^2_{1,e,(G\times Y)\setminus \UN^{m-5}(T\times S)}.\]
Repeated application of this argument gives for $\tilde{\be}:=2\tilde{C}/(2\tilde{C}+1)<1$ together with the stability of $\VQ_{T\times S, \infty}$ that
\[\|\Vphi\|^2_{1,e,(G\times Y)\setminus \UN^m(T\times S)}\leq \tilde{\be}^{\lfloor m/5\rfloor}\|\Vphi\|^2_{1,e}\leq C_{\VQ}\tilde{\be}^{\lfloor m/5\rfloor}\|\Vv_{H_c, h_c}\|^2_{1,e,T\times S},\]
which gives the assertion after some algebraic manipulations.
\end{proof}

As the localized correctors $\VQ_m$ are the Galerkin approximations of the idealized correctors $\VQ_\infty$, the decay property carries over to $\VQ_m$. 
This is the main observation for the proof of Theorem \ref{thm:decaycorrec}.

\begin{proof}[Proof of Theorem \ref{thm:decaycorrec}]
For $m\geq 3$, we define the cut-off functions $\eta_0\in V_{H_c}^1$, $\eta_1\in L^2(\Om; \widetilde{V}_{h_c}^1(Y^*))$ and $\eta_2\in L^2(\Om; V_{h_c}^1(D))$ via
\begin{align*}
\eta_0 &=0 \quad\text{ in }G\setminus\UN^{m-1}(T)&\text{ and }&&\eta_0&=1 \quad\text{ in } \UN^{m-2}(T),\\*
\eta_1 &=0 \quad\text{ in }(\Om\times Y^*)\setminus\UN^{m-1}(T\times S_1)&\text{ and }&&\eta_1&=1 \quad\text{ in } \UN^{m-2}(T\times S_1),\\*
\eta_2 &=0 \quad\text{ in }(\Om\times D)\setminus\UN^{m-1}(T\times S_2)&\text{ and }&&\eta_2&=1 \quad\text{ in }\UN^{m-2}(T\times S_2),
\end{align*}
where again $\eta_1$ and $\eta_2$ are w.l.o.g.\ piece-wise $x$-constant.
The cut-off functions $\eta_i$ satisfy the bounds \eqref{eq:gradeta}. Set again $\Veta:=(\eta_0, \eta_1, \eta_2)$. 
As already discussed, $\VQ_{T\times S, m}$ can be interpreted as Galerkin approximation of $\VQ_{T\times S, \infty}$ in the discrete subspace $\VW_{H,h}(G_T\times Y_S)\subset \VW_{H,h}$. Hence, C{\'e}a's Lemma gives for any $\Vw_{H,h}\in \VW_{H,h}(G_T\times Y_S)$
\[\|(\VQ_{T\times S, \infty}-\VQ_{T\times S, m})\Vv\|^2_{1,e}\leq C\|\VQ_{T\times S, \infty}\Vv-\Vw_{H, h}\|^2_e.\]
We choose $\Vw_{H,h}:=(1-\VI_{H_c, h_c})\VCI_{H,h}(\Veta\VQ_{T\times S, \infty}\Vv)\in \VW_{H,h}(G_T\times Y_S)$ and obtain with the identity $\VI_{H_c, h_c}\VQ_{T\times S, \infty}\Vv=0$, the estimate \eqref{eq:L2IHcIh}, the approximation and stability  estimates \eqref{eq:intpol} and \eqref{eq:nodalintpolh1stable}, the resolution condition \eqref{eq:resolassptcorrec} and estimate \eqref{eq:nablaeta} that
\begin{align*}
&\!\!\!\!\|(\VQ_{T\times S, \infty}-\VQ_{T\times S, m})\Vv\|_{1,e}^2\\*
&\leq C\|\VQ_{T\times S, \infty}\Vv-(1-\VI_{H_c, h_c})\VCI_{H,h}(\Veta\VQ_{T\times S, \infty}\Vv)\|_e^2\\
&=C\|(1-\VI_{H_c, h_c})\VCI_{H,h}(\VQ_{T\times S, \infty}\Vv-\Veta\VQ_{T\times S, \infty}\Vv)\|^2_{e, (G\times Y)\setminus\{\Veta=1\}}\\
&\leq C\|(1-\Veta)\VQ_{T\times S, \infty}\Vv\|^2_{1,e,\UN((G\times Y)\setminus\{\Veta=1\})}\\*
&\leq C\|\VQ_{T\times S, \infty}\Vv\|^2_{1, e, \UN((G\times Y)\setminus \{\Veta=1\})}.
\end{align*}
Note that $\UN((G\times Y)\setminus \{\Veta=1\})=(G\times Y)\setminus \UN^{m-3}(T\times S)$. Together with Proposition \ref{prop:decayQinfty}, this proves \eqref{eq:decaydiffcorrecT}.

Define $\Vz:=(\VQ_{\infty}-\VQ_m)\Vv$ and $\Vz_{T\times S}:=(\VQ_{T\times S, \infty}-\VQ_{T\times S, m})\Vv$. The ellipticity from Lemma \ref{lem:wellposedcorrec} yields
\[\|\Vz\|^2_{1,e}\leq C\Bigl| \sum_{(T, S_1, S_2)\in \CT_{H_c}\times \CT_{h_c}(Y^*)\times \CT_{h_c}(D)}\CB(\Vz, \Vz_{T\times S})\Bigr|.\]
We define the cut-off functions $\eta_0\in V_{H_c}^1$, $\eta_1\in L^2(\Om; \widetilde{V}_{h_c}^1(Y^*))$, and $\eta_2\in L^2(\Om; V_{h_c}^1(D))$ via
\begin{align*}
\eta_0 &=1 \quad\text{ in }G\setminus\UN^{m+2}(T)&\text{ and }&&\eta_0&=0 \quad\text{ in } \UN^{m+1}(T),\\
\eta_1 &=1 \quad\text{ in }(\Om\times Y^*)\setminus\UN^{m+2}(T\times S_1)&\text{ and }&&\eta_1&=0 \quad\text{ in } \UN^{m+1}(T\times S_1),\\
\eta_2 &=1 \quad\text{ in }(\Om\times D)\setminus\UN^{m+2}(T\times S_2)&\text{ and }&&\eta_2&=0 \quad\text{ in }\UN^{m+1}(T\times S_2),
\end{align*}
where again $\eta_1$ and $\eta_2$ are w.l.o.g.\ piece-wise $x$-constant.
The cut-off functions fulfill \eqref{eq:gradeta} and we set $\Veta:=(\eta_0, \eta_1, \eta_2)$. 
For any $(T, S_1, S_2)\in \CT_{H_c}\times \CT_{h_c}(Y^*)\times \CT_{h_c}(D)$ we have $(1-\VI_{H_c, h_c})\VCI_{H,h}(\Veta\Vz)\in \VW_{H,h}$ with support outside $G_T\times Y_S$. 
Hence, we deduce with $\Vz=\VCI_{H,h}\Vz$ that
\[\CB(\Vz, \Vz_{T\times S})=\CB(\VCI_{H,h}(\Vz-\Veta\Vz), \Vz_{T\times S})+\CB(\VI_{H_c, h_c}\VCI_{H,h}(\Veta\Vz), \Vz_{T\times S}).\]
The function $\Vz-\VCI_{H,h}(\Veta \Vz)$ vanishes on $\{\Veta =1\}$. Thus, the first term on the right-hand satisfies
\[|\CB(\VCI_{H, h}(\Vz-\Veta\Vz), \Vz_{T\times S})|\leq C_B\|\VCI_{H,h}(\Vz-\Veta\Vz)\|_{e,\supp(1-\Veta)}\,\|\Vz_{T\times S}\|_e.\]
The $L^2$- and $H^1$-stability of $\VCI_{H,h}$ on piece-wise polynomials gives together with the estimate \eqref{eq:nablaeta} applied to the cut-off function $1-\Veta$ 
\[\|\VCI_{H,h}(\Vz-\Veta\Vz)\|_{e, \supp(1-\Veta)}\leq C\|\Vz\|_{e, \UN(\supp(1-\Veta))}.\] 
Furthermore, $\VI_{H_c, h_c}\VCI_{H,h}(\Veta\Vz)$ vanishes on $(G\times Y)\setminus \UN(\supp(1-\Veta))$. Therefore, we infer from the stability \eqref{eq:intpolenergystable} of $\VI_{H_c, h_c}$, the stability of $\VCI_{H,h}$ as before, and estimate \eqref{eq:nablaeta} that
\[|\CB(\VI_{H_c, h_c}\VCI_{H,h}(\Veta\Vz), \Vz_{T\times S})|\leq C\|\Vz\|_{e, \UN^2(\supp(1-\Veta))}\,\|\Vz_{T\times S}\|_e.\]
Because of the resolution condition \eqref{eq:resolassptcorrec} and Lemma \ref{lem:wellposedcorrec}, it holds $\|\Vz\|_e\leq C\|\Vz\|_{1,e}$.
Summing up over  $(T, S_1, S_2)\in \CT_{H_c}\times \CT_{h_c}(Y^*)\times \CT_{h_c}(D)$ yields with the Cauchy-Schwarz inequality and the finite overlap of patches that
\begin{align*}
\|\Vz\|^2_{1,e}&\leq C\sum_{(T, S_1, S_2)\in \CT_{H_c}\times \CT_{h_c}(Y^*)\times \CT_{h_c}(D)}\|\Vz\|_{1,e,\UN^2(\supp(1-\Veta))}\, \|\Vz_{T\times S}\|_e\\
&\leq C\bigl(\sqrt{C_{\ol,m, G}}+\sqrt{C_{\ol, m, Y}}\bigr)\|\Vz\|_{1,e}\sqrt{\sum_{(T, S_1, S_2)\in \CT_{H_c}\times \CT_{h_c}(Y^*)\times \CT_{h_c}(D)}\|\Vz_{T\times S}\|^2_e}.
\end{align*}
Combining the last estimate with \eqref{eq:decaydiffcorrecT} concludes the proof.
\end{proof}

\section*{Conclusion}
In this paper, we presented a Localized Orthogonal Decomposition in Petrov-Galerkin formulation for the two-scale Helmholtz-type problems coming from homogenization, see for instance \cite{OV16hmmlod1}. 
Under the natural assumption of a few degrees of freedom per wave length, $k(H_c+h_c)\lesssim1$, and suitably chosen oversampling patches, $m\approx\log(k)$, the method is stable and quasi-optimal without any further restrictions on the mesh width. 
Thereby, this work gives the theoretical foundation and justification for an application of the LOD to the two-scale setting in case the resolution condition poses a practical restriction for the direct  discretization.
We have demonstrated how the periodicity in the function spaces and the coupling of different spaces can be tackled in the LOD framework.
Furthermore, this paper underlines the significance of viewing the Heterogeneous Multiscale Method as a direct discretization with numerical quadrature since only this viewpoint makes the additional application of the LOD possible.

\section*{Acknowledgments}
Financial support by the German research Foundation (DFG) under the project ``OH 98/6-1: Wave propagation in periodic structures and negative refraction mechanisms'' is gratefully acknowledged. 
The authors would like to thank P.\ Henning, A.\ Lamacz, and B.\ Schweizer for fruitful discussions on the subject.

\end{document}